\documentclass[11pt]{article}

\makeatletter
\newcommand{\doublewidetilde}[1]{{%
  \mathpalette\double@widetilde{#1}%
}}
\newcommand{\double@widetilde}[2]{%
  \sbox\z@{$\m@th#1\widetilde{#2}$}%
  \ht\z@=.9\ht\z@
  \widetilde{\box\z@}%
}
\usepackage{xpatch}
\usepackage{alltt}
\makeatletter
\g@addto@macro\bfseries{\boldmath}
\makeatother
\usepackage{bm}
\usepackage{extpfeil}
\usepackage{setspace}
\usepackage{tikz}
\usepackage{siunitx}

\usepackage{blindtext, rotating}
\usepackage{booktabs}
\usepackage{array}
\newcolumntype{C}[1]{>{\centering\arraybackslash}p{#1}}
\usepackage{color}
\usepackage{float}
\usetikzlibrary{matrix,positioning}
\usepackage{mathtools}

\makeatother

\usepackage{amsmath,tabu}
\usepackage{tikz-cd}
\usepackage{longtable}
\usepackage[utf8]{inputenc}
\usepackage{ctable}
\usepackage{eufrak}
\usepackage{pdflscape}
\usepackage{array}
\newcolumntype{P}[1]{>{\centering\arraybackslash}p{#1}}
\usepackage{enumerate}
\usepackage{amsthm}
\usepackage{amsfonts}
\usepackage{multicol}
\usepackage{array}
\newcolumntype{L}{>{$}l<{$}}
\usepackage{hyperref}
\usepackage[margin=1in]{geometry}
\usepackage{enumitem}  
\begingroup
\catcode`\&=13
\gdef\pampmatrix{%
  \begingroup
  \let&=\amsamp
  \begin{pmatrix}%
}
\gdef\endpampmatrix{\end{pmatrix}\endgroup}
\endgroup
\usepackage{stackengine}
\newtheorem{theorem}{Theorem}[section]

\newtheorem{corollary}[theorem]{Corollary}

\newtheorem{lemma}[theorem]{Lemma}
\newtheorem{proposition}[theorem]{Proposition}

\usepackage{amssymb}

\theoremstyle{definition}

\newtheorem{remark}[theorem]{Remark}
\newtheorem{definition}[theorem]{Definition}
\stackMath
\usepackage{amssymb}
\usepackage{graphicx}
\usepackage{hyperref}
\makeatletter
\renewcommand*\env@matrix[1][*\c@MaxMatrixCols c]{%
  \hskip -\arraycolsep
  \let\@ifnextchar\new@ifnextchar
  \array{#1}}
\makeatother
\newcommand{\A}[1]{\mathrm{A}_{#1}} 
\renewcommand{\S}[1]{\mathrm{S}_{#1}} 
 
\renewcommand{\sup}{\mathrm{Supp}}

\newcommand{\Sym}[1]{\mathrm{Sym}({#1})}

\newcommand{\PSL}{\mathrm{PSL}} 
 
\newcommand{\Out}{\mathrm{Out}} 
\newcommand{\Sp}{\mathrm{Sp}} 
\newcommand{\PSp}{\mathrm{PSp}} 

\newcommand{\PGamL}{\mathrm{P} \Gamma \mathrm{L}} 
\newcommand{\POm}{\mathrm{P\Omega}} 

\newcommand{\GL}{\mathrm{GL}} 
 
\newcommand{\PGL}{\mathrm{PGL}} 
\newcommand{\Syl}[1]{\mathrm{Syl}} 
\renewcommand{\max}{\mathrm{max}} 
\renewcommand{\k}{k} 
\renewcommand{\l}{l}
\newcommand{\s}[1]{^{(#1)}}
\newcommand{\I}{\mathrm{I}}
\newcommand{\RC}{\mathrm{RC}}
\newcommand{\F}{\mathbb{F}}

\renewcommand{\H}{\mathrm{H}}
\newcommand{\M}{\mathrm{M}}
\renewcommand{\a}{a}
\renewcommand{\b}{b}
\newcommand{\B}{\mathrm{B}}
\newcommand{\bb}{\mathrm{b}}
\newcommand{\PG}{\mathcal{PG}}
\newcommand{\Soc}{\mathrm{Soc}}
\renewcommand{\c}{c}
\renewcommand{\d}{f}
\newcommand{\n}{d}
\renewcommand{\t}{n}
\newcommand{\C}{C}
\usepackage{bigfoot}
\usepackage[pagewise]{lineno}
\begin{document}
\title{{\Large}{\textbf{On relational complexity and base size of finite primitive groups}}}

\author{Veronica Kelsey and Colva M. Roney-Dougal \thanks{ 
\textit{Acknowledgements:} The authors would like to thank Nick Gill for encouraging us to look at this problem and Peter Cameron for his helpful remarks, which improved the paper. \newline \indent
\textit{Key words:} permutation group; base size; relational complexity; computational complexity \newline \indent
\textit{MSC2020:} 20B15; 20B25; 20E32; 20-08}}
\maketitle
\begin{abstract} 
In this paper we show that if $G$ is a primitive subgroup of $\S{n}$ that is not large base, then any irredundant base for $G$ has size at most $5 \log n$. This is the first logarithmic bound on the size of an irredundant base for such groups, and is best possible up to a small constant.
As a corollary, the relational complexity of $G$ is at most $5 \log n+1$, and the maximal size of a minimal base and the height are both at most $5 \log n.$ Furthermore, we deduce that a base for $G$ of size at most $5 \log n$ can be computed in polynomial time.
\end{abstract}
\section{Introduction}\label{intro}

Let $\Omega$ be a finite set. A \textit{base} for a subgroup $G$ of $\Sym{\Omega}$ is a sequence $\Lambda = (\omega_1, \ldots, \omega_{\l})$ of points of $\Omega$ such that $G_{\Lambda} = G_{\omega_1, \ldots, \omega_{\l}}=1$. 
The \textit{minimum base size}, denoted $\bb(G,\Omega)$ or just $\bb(G)$ if the meaning is clear, is the minimum length of a base for $G$. Base size has important applications in computational group theory; see, for example, \cite{Sims} for the importance of a base and strong generating set.

In \cite{Liebeck} Liebeck proved the landmark result that with the exception of one family of groups, if $G$ is a primitive subgroup of $\S{n}=\Sym{\{1,\ldots, n\}}$ then $\bb(G) < 9 \log \t$. The exceptional family are called \textit{large-base} groups: product action or almost simple groups whose socle is one or more copies of the alternating group $\A{r}$ acting on $k$-sets. In \cite{Mariapia} Moscatiello and Roney-Dougal improve this bound, and show that if $G$ is not large base then either $G=\M_{24}$ in its 5-transitive action of degree 24, or $\bb(G) \leq \lceil \log n \rceil +1$. Here and throughout all $\log$s are of base 2.

We say that a base $\Lambda = (\omega_1, \ldots, \omega_k)$ for a permutation group $G$ is \textit{irredundant} if
$$G > G_{\omega_1} > G_{\omega_1,\omega_2} > \cdots > G_{\omega_1, \ldots, \omega_{\l}} =1.$$
If no irredundant base is longer than $\Lambda$, then $\Lambda$ is a \textit{maximal} irredundant base and we denote the length of $\Lambda$ by $\I(G, \Omega)$ or $\I(G)$. \smallskip

From Liebeck's $9\log n$ bound on base size, a straightforward argument (see Lemma \ref{RCbound}) shows that if $G$ is a primitive non-large-base subgroup of $\S{\t}$, then $\I(G) \leq 9\log^2 \t$.
However, in \cite{GLS} Gill, Lod\'{a} and Spiga conjecture that for such groups $G$ there exists a constant $c$ such that $\I(G) \leq c \log \t$. They show that for some families of groups the conjecture holds with $c=7$.
Our main result establishes this conjecture, whilst also improving the constant.
\begin{theorem}\label{main}
Let $G $ be a primitive subgroup of $\S{n}$. If $G$ is not large base, then 
$$\I(G) < 5 \log n.$$
\end{theorem}

There are infinitely many primitive groups for which the maximal irredundant base size is greater than $\lceil \log \t \rceil +1$. For example, if $d \geq 5$, $G=\PGL_{\n}(3)$ and $\Omega$ is the set of 1-spaces of $\F_3^{\n}$, then by Theorem \ref{PGLthm} $\I(G, \Omega) = 2\n-1>\lceil \log \t \rceil +1$. Hence, up to a small constant the bounds in Theorem \ref{main} are the best possible.\medskip

Relational complexity has been extensively studied in model theory, see for example \cite{lach}. A rephrasing of the definition, to make it easier to work with for permutation groups, was introduced more recently in \cite{cherlin}. For an excellent discussion and more context, see \cite{BinaryPaper}.
Let $\k,\l \in \mathbb{N}$ with $\k \leq \l$, and let $\Lambda = (\lambda_1, \ldots, \lambda_{\l}), \Sigma = (\sigma_1, \ldots, \sigma_{\l}) \in \Omega^{\l}$.
We say that $\Lambda$ and $\Sigma$ are \textit{$\k$-subtuple complete} with respect to a subgroup $G$ of $\Sym{\Omega}$, and write $\Lambda \sim_{\k} \Sigma$, if for every subset of $\k$ indices $ i_1, \ldots, i_{\k}$ there exists $g \in G$ such that $(\lambda_{i_1}^g, \ldots, \lambda_{i_{\k}}^g) = (\sigma_{i_1}, \ldots, \sigma_{i_{\k}}).$
The \textit{relational complexity} of $G$, denoted $\RC(G)$, is the smallest $\k$ such that for all $\l \geq \k$ and all $\Lambda, \Sigma \in \Omega^{\l}$, if $\Lambda \sim_{\k} \Sigma$ then $\Lambda \in \Sigma^G$. In \cite{Cherlin2} Cherlin gives examples of groups with relational complexity 2, called \textit{binary groups}, and conjectures that this list is complete. In a dramatic breakthrough, Gill, Liebeck and Spiga have just announced a proof of this conjecture, see \cite{BinaryPaper}.

Let $\Lambda$ be a base for a permutation group $G$. Then $\Lambda$ is \textit{minimal} if no proper subsequence of $\Lambda$ is a base. We denote the maximum size of a minimal base by $\B(G)$.
The \textit{height} $\H(G)$ of $G$, is the size of the largest subset $\Delta$ of $\Omega$ with the property that $G_{(\Gamma)} \neq G_{(\Delta)}$ for each $\Gamma \subsetneq \Delta$.
The following key lemma relates all of these group statistics studied in this paper.
\begin{lemma}\label{RCbound}\emph{\cite[Equation 1.1 and Lemma 2.1]{GLS}} Let $G$ be a subgroup of $\S{\t}$. Then 
$$ \bb(G) \leq \B(G) \leq \H(G) \leq \I(G) \leq \bb(G) \log \t,$$
and
$$ \RC(G) \leq \H(G)+1.$$
\end{lemma}
In \cite{GLS}, Gill, Lod\'{a} and Spiga prove that if $G \leq \S{\t}$ is primitive and not large base, then $\H(G) < 9 \log \t$ and so
$\RC(G) < 9 \log \t+1$ and $\B(G) < 9 \log \t$.

It will follow immediately from Theorem \ref{main} and Lemma \ref{RCbound} that we can tighten the all of these bounds.
\begin{corollary}\label{main2}
Let $G$ be a primitive subgroup of $\S{\t}$. If $G$ is not large base then 
$$ \RC(G) < 5\log \t +1, \;\;\;\;   \B(G) < 5 \log \t,  \;\;\;\; \text{and} \;\;\;\; \H(G) < 5 \log \t.$$
\end{corollary}

Blaha proved in \cite{Blaha} that the problem of computing a minimal base for a permutation group $G$ is NP-hard. Furthermore, he showed that the obvious greedy algorithm to construct an irredundant base for $G$ produces one of size $O(\bb(G) \log \log n)$. Thus if $G$ is primitive and not large base, it follows from Liebeck’s result that in polynomial time one can construct a base of size $O(\log n \log \log n)$. 
Since an irredundant base size can be computed in polynomial time (see for example \cite{Sims}), we get the following corollary, which improves this bound to the best possible result, up to a constant.
\begin{corollary} Let $G$ be a primitive subgroup of $\S{\t}$ which is not large base. Then a base for $G$ of size at most $5 \log \t$ can be constructed in polynomial time.
\end{corollary}
(We note that using the bound on $\B(G)$ from \cite{GLS}, a very slightly more complicated argument would yield a similar result, but with $9 \log \t$ in place of $5 \log \t$.)\medskip

The paper is structured as follows. In Section \ref{prelim} we prove some preliminary lemmas about $\I(G)$.
In Section \ref{pglmspaces} we give upper and lower bounds on the size of an irredundant base for $\PGL_{\n}(q)$ acting on subspaces of $\F_q^{\n}$, which differ by only a small amount. In Section \ref{ASgroups} we prove a result which is a slight strengthening of Theorem \ref{main} for almost simple groups. Finally, in Section \ref{PAgroups} we complete the proof of Theorem \ref{main}.
\section{Preliminary bounds on group statistics}\label{prelim}
Here we collect various lemmas about bases, and about the connection between $\I(G)$ and other group statistics.


For a subgroup $G$ of $\Sym{\Omega}$ and a fixed sequence $(\omega_1, \ldots, \omega_l)$ of points from $\Omega$, we let $G^{(i)}=G_{\omega_1, \ldots, \omega_i}$ for $0 \leq i \leq l$, so $G^{(0)} = G$. Furthermore, the maximum length of a chain of subgroups in $G$ is denoted by $\ell(G)$. 
\begin{lemma}\label{fands} Let $G$ be a subgroup of $\S{\t}$. 
\begin{enumerate}[label=\rm{(\roman*)}]
\item If $G$ is insoluble, then $\I(G) < \log |G|-1$. \label{insol}
\item If $G$ is transitive and $n\geq 5$, then $\I(G) \leq \log |G|-1.$ \label{sneaky}
\item If $G$ is transitive and $b = \bb(G)$, then $\I(G) \leq (b-1)\log \t +1 .$ \label{fourth}
\end{enumerate}
\end{lemma}
\begin{proof}
Let $a$ be the number of prime divisors of $|G|$, counting multiplicity. Since $G$ is insoluble there exists a prime greater than $2^2$ dividing $|G|$, and so $|G| > 2^{a+1}$. It is clear that $\I(G) \leq \ell(G) \leq a$, and so Part (i) follows, and we assume from now on that $G$ is transitive.

Let ${\l}=\I(G)$ with a corresponding base $\Lambda = ( \omega_1, \ldots,  \omega_{\l})$. 
Since $G$ is transitive, $[G^{(0)}:G^{(1)}] =\t$ by the Orbit-Stabiliser Theorem. From $ [G\s{i-1}:G\s{i}] \geq 2 $ for $2 \leq i \leq {\l}$, it follows that $|G| \geq 2^{l-1}n$. Hence if $n \geq 5$ then $|G| \geq 2^{l-1}\cdot 5>2^{l+1}$. Therefore by taking logs Part (ii) follows.

Similarly, $|G| < n^b$, and so $ 2^{{\l}-1}\t \leq |G| \leq \t^b$. Hence
$$l-1+\log \t = \log (2^{l-1}\t) \leq \log |G| \leq b\log \t,$$
and so
$${\l} \leq b \log \t - \log \t +1 = (b-1)\log \t+1 $$
and Part (iii) follows.
\end{proof}

\begin{lemma}\label{2.2.6} Let $G$ be a subgroup of $\Sym{\Omega}$, let $\l \geq 1$ and let $\Lambda =(\lambda_1, \ldots, \lambda_{\l}) \in \Omega^{\l}$. Then there exists a subsequence $\Sigma$ of $\Lambda$ such that $\Sigma$ can be extended to an irredundant base and $G_{\Sigma} = G_{\Lambda}$.
\end{lemma}
\begin{proof} The sequence $\Lambda$ cannot be extended to an irredundant base if and only if there exists a subsequence $\lambda_i, \ldots, \lambda_{i+j}$ of $\Lambda$ with $j \geq 1$ such that 
$$ G\s{i} = G\s{i+1} = \cdots = G\s{i+j}. $$
Let $\Sigma$ be the subsequence of $\Lambda$ given by deleting all such $\lambda_{i+1}, \ldots, \lambda_{i+j}$. Since $G\s{i} = G\s{i+j}$ it follows that $G_{\Lambda} = G_{\Sigma}$.
\end{proof}
The following describes the relationship between the irredundant base size of a group and that of a subgroup. 
\begin{lemma}\label{3comb} Let $H$ and $G$ be subgroups of $ \S{\t},$ with $H \leq G$. Then the following hold.
\begin{enumerate}[label=\rm{(\roman*)}]
\item $\I(H) \leq \I(G)$. \label{subgroup}
\item If $H \trianglelefteq G$ then $\I(G) \leq \I(H) +\ell(G/H).$ \label{normal}
\item If $H \trianglelefteq G$ and $[G:H]$ is prime, then 
$\I(H) \leq \I(G) \leq \I(H)+1.$ \label{2.3.8}
\end{enumerate}
\end{lemma}
\begin{proof}An irredundant base for $H \leq G$ can be extended to an irredundant base for $G$, so Part (i) is clear. Part (ii) is \cite[Lemma 2.8]{GLS} and Part (iii) follows immediately from Parts (i) and (ii). \end{proof}
\section{Groups with socle $\PSL_d(q)$ acting on subspaces}\label{pglmspaces}
Throughout this section, let $q = p^f$ for $p$ a prime and $f \geq 1$, let $V$ be a $\n$-dimensional vector space over $\F_q ,$ let $\Omega=\PG_m(V)$ be the set of all $m$-dimensional subspaces of $V$, and let $\t = |\Omega|$.
In this section we begin by proving Theorem \ref{PGLthm} which bounds $\I(\PGL_{\n}(q), \Omega)$ in terms of $\n$ and $m$.
By finding lower bounds on $|\Omega|$ we then prove Proposition \ref{g01}, which bounds $\I(\PGamL_{\n}(q), \Omega)$ in terms of $|\Omega|$.
\subsection{Bounds as a function of $\n$ and $m$}
In this subsection we prove the following theorem, which in the case $m=1$ recovers the lower bounds found by Lod\'{a} in \cite{Bianca}. 
\begin{theorem}\label{PGLthm}
Let $\PGL_{\n}(q)$ act on $\Omega$. Then 
$$\I(\PGL_{\n}(q)) \leq (m+1)\n-2m+1, $$
and 
$$\I(\PGL_{\n}(q)) \geq \begin{cases}
m\n-m^2+1 & \text{ if } q = 2,\\
(m+1)\n -m^2 & \text{ if } q \neq 2.
\end{cases}$$
\end{theorem}

We begin by proving the upper bound in Theorem \ref{PGLthm}.
Let $M=M(V)$ be the algebra of all linear maps from $V$ to itself. Furthermore, let $\omega_0 = \langle \underline{0} \rangle$, let $\l >1$ be an integer, let $ \Lambda = (\omega_1, \omega_2, \ldots, \omega_{\l}) \in \Omega^{\l}$ and for $0 \leq \k \leq \l$, let
$$M_{\k} = \{g \in M \; | \; \omega_i g \leq \omega_i \; \text{for} \; 0 \leq i \leq \k\}, \text{ so that } M_0=M.$$
For $0 \leq k \leq l-1$, it is easily verified that $M_{k+1}$ is a subspace of $M_{k}$. Now assume in addition that
\begin{equation}\label{rewrite} M_0>M_1> \cdots >M_l = \F_q I
\end{equation}
with $\l$ as large as possible. Fix a basis $\langle e_1, \ldots, e_{\n} \rangle$ of $V$ which first goes through $\omega_1 \cap \omega_2$, then extends to a basis of $\omega_1$, and then for each $\k \geq 2$ extends successively to a basis of $\langle \omega_1, \ldots, \omega_{\k} \rangle$. Therefore there exist integers
$$ m = \a_1 \leq \cdots \leq \a_{\l} = d \; \text{ such that } \; \langle e_1, \ldots, e_{\a_{\k}} \rangle = \langle \omega_1, \ldots, \omega_k \rangle.$$
Since $\omega_0= \langle \underline{0} \rangle$, we may let $\a_0=0$. From now, on we identify $M$ with the algebra of $\n \times \n$ matrices over $\F_q$ with respect to this basis.\medskip

We will show that $\l \leq (m+1)\n-2m+1$, from which the upper bound in Theorem \ref{PGLthm} will follow. For $0 \leq k \leq l-1$, let
$$\d_{\k} = \dim(M_{\k}) - \dim(M_{\k+1})$$
and let $b_k=a_{k+1}-a_{k}$ so that $0 \leq b_k \leq m$.
In the following lemmas we consider the possible values of $\d_{\k}$ based on $b_{\k}$.
\begin{lemma}\label{space1+2} Let $\d_{\k}$ and $b_{\k}$ be as above. \begin{enumerate}[label=\rm{(\roman*)}]
\item The dimension of $M_1$ is $ \n^2-m(\n-m),$ and so $\d_0 = m(\n-m)$.
\item $\d_1 =\b_1(\n-\b_1) $.
\end{enumerate}
\end{lemma}
\begin{proof}
First consider Part (i). Since $\omega_1 = \langle e_1, \ldots, e_m \rangle$ it follows that $g=(g_{ij}) \in M_1$ if and only if $e_ig \in \omega_1$ for $1 \leq i \leq m$. Equivalently, $g_{ij} = 0$ for $1 \leq i \leq m$ and $m+1 \leq j \leq \n$. Hence $\dim(M_1) =  \n^2-m(\n-m)$, and the final claim follows from $\dim(M_0)=\n^2$.

Now consider (ii). The subspace $M_2$ contains all matrices of shape
$$\begin{pmatrix}
x_1 & 0 & 0 & 0\\
x_2 & x_3 & 0 & 0\\
x_4 & 0 & x_5 & 0\\
y_1 & y_2 & y_3 & y_4 
\end{pmatrix} $$
where $x_1,$ $x_3$ and $x_5$ are square with $m-\b_1$, $\b_1$ and $\b_1$ rows respectively.
Hence
\begin{align*}
\dim(M_2) &= (m-\b_1)^2+2\b_1(m-\b_1)+2\b_1^2+(\n-m-\b_1)\n\\
&=\n^2-m(\n-m)-\b_1(\n-\b_1),
\end{align*}
and the result follows from Part (i).
\end{proof}
\begin{lemma}\label{dimt} Let $\k \geq 2$. Then $\d_{\k} \geq \max\{1, \b_{\k}(\n-m)\}$.
\end{lemma}
\begin{proof} For $0 \leq \k \leq \l$ we define two subspaces of $M_{\k}$, namely
\begin{align*}
X_{\k} &= \{g \in M_{\k} \; | \; e_ig=0 \text{ for } \a_{\k}+1 \leq i \leq \n \}  \text{ and}\\
Y_{\k} &= \{g \in M \; | \; e_ig=0 \text{ for } 1 \leq i \leq  \a_{\k}\}.
\end{align*} 
We begin by showing that 
\begin{equation}\label{Mt} M_{\k} = X_{\k} \oplus Y_{\k} \;\;\;  \text{and} \;\;\; \dim(Y_{\k})=\n(\n-\a_{\k}). 
\end{equation}
  
It is clear that $X_{\k} \cap Y_{\k} = \{0_M\}$.
Let $g =(g_{ij}) \in M_k$. Then there exist $x=(x_{ij}),y=(y_{ij}) \in M$ with $x_{ij}=g_{ij}$ and $y_{ij}=0$ for $i \leq \a_{\k}$, and $x_{ij}=0$ and $y_{ij} = g_{ij}$ for $i \geq \a_{\k}+1$. Then $g=x+y$ with $x \in X_{\k}$ and $y \in Y_{\k}$, hence $M_{\k} = X_{\k} \oplus Y_{\k}$. 
Since $g \in Y_{\k}$ if and only if $g_{ij}=0$ for $i \leq \a_{\k}$, it follows that $\dim({Y_{\k}})=\n(\n-\a_{\k})$. Hence \eqref{Mt} holds.\medskip

Our assumption that $M_{k} > M_{k+1}$ implies that $\d_{\k} \geq 1$, so we may assume that $b_{k} \geq 1$.
By \eqref{Mt}, 
\begin{align*}
\d_{\k} &=\dim(M_{\k})-\dim(M_{\k+1})\\
&= \big( \dim(X_{\k})+\dim(Y_{\k}) \big) -\big( \dim(X_{\k+1})+\dim(Y_{\k+1}) \big) \\
&= \dim(X_{\k})-\dim(X_{\k+1})+\n(\n-\a_{\k}) -\n(\n-\a_{\k+1}) \\
&= \dim(X_{\k})-\dim(X_{\k+1})+\b_{\k} \n. 
\end{align*}
We now bound $\dim(X_{\k}) - \dim(X_{\k+1} )$. By choice of basis $\omega_{k+1}=\langle u_1, \ldots, u_{m-b_k}, e_{a_k+1}, \ldots, e_{a_{k}+b_{k}} \rangle$ for some $u_1, \ldots, u_{m-\b_k} \in \langle \omega_1, \ldots, \omega_k \rangle$. Hence if $v \in   \{ e_{\a_{\k}+1}, \ldots , e_{\a_{\k}+\b_{\k}}\}$ then $\langle v \rangle M_{\k+1} \leq \omega_{\k+1}$. Therefore $\langle v M_{\k+1} \rangle$ has dimension at most $m$, and so $\dim(X_{\k+1}) \leq \dim(X_{\k}) +\b_{\k} m$. Hence
$$\d_{\k} =\dim(X_k)-\dim(X_{k+1})+\b_{\k}\n \geq - \b_{\k} m +\b_{\k} \n =\b_{\k} (\n-m).$$
\end{proof}

\noindent \textit{Proof of upper bound of Theorem \ref{PGLthm}.} We shall show that $ \l \leq (m+1)\n-2m+1$, from which the result  will follow, since $\I(\PGL_{\n}(q), \Omega)=\I(\GL_{\n}(q), \Omega),$ and $\GL_d(q)$ is a subgroup of $M$.

For $0 \leq \b \leq m$, let  
$$\C_{\b} = \big\{\k \in \{0, \ldots, \l-1 \} \; \big| \; \b_{\k}=\b \big\}.$$
and let $\c_{\b} = |\C_{\b}|$. 
Then
\begin{equation}\label{max} \l = \sum_{\b=0}^{m} \c_{\b}. 
\end{equation}
Since $a_l=d$ and $\a_0=0$ it follows that
\begin{equation}\label{c0c1}
\n = \a_{\l} - \a_0 = \sum_{\k=0}^{\l-1} (\a_{\k+1}-\a_{\k})= \sum_{k=0}^{l-1}\b_{k} =\sum_{\b=0}^m \b\c_{\b} =\sum_{\b=1}^m \b \c_{\b}.
\end{equation}
Since $\a_1=m$ and $\a_0=0$ it follows that $b_0=m$, so $0 \in \C_m,$ and $\c_m \geq 1$. Since $\omega_1 \neq \omega_2$ it follows that $\b_1 \neq 0$, and $1 \in \C_{\b_1}$, so
\begin{equation}\label{crcm} \c_{\b_1} \geq 1 \;\;\;\;\; \text{and} \;\;\;\;\; \c_m \geq 1+\delta_{m\b_1}.
\end{equation}
Lemmas \ref{space1+2} and \ref{dimt} yield
\begin{equation}\label{7}
\d_0=m(\n-m), \; \d_1=\b_1(\n-\b_1)=\b_1(m-\b_1)+\b_1(\n-m), \; \text{and} \; \d_{\k} \geq \max\{1,b_{k}(\n-m)\} \text{ for $k \geq 2.$}
\end{equation}

Since $M_0=M$ and $M_{\l} = \F_q I$ it follows from the definition of $\d_{\k}$ that
\begin{align*}
\n^2-1 &= \dim(M_0)-\dim(M_{\l})\\
&= \sum_{\k=0}^{\l-1} \Big( \dim(M_{\k})-\dim(M_{\k+1}) \Big)\\
&= \sum_{\k=0}^{\l-1} \d_{\k}\\
&= \sum_{\k \in \C_0} \d_{\k} + \d_1 + \sum_{\k \in \C_{\b_1} \backslash\{1\}} \d_{\k}  + \sum_{\k \notin \C_0 \cup \C_{\b_1}} \d_{\k}\\
&\geq \sum_{\k \in \C_0} 1 +\b_1(m-\b_1)+\b_1(\n-m) + \sum_{\k \in \C_{\b_1} \backslash\{1\}} \b_1(\n-m)  + \sum_{\k \notin \C_0 \cup \C_{\b_1}} \b_{\k}(\n-m) \;\;\;\;\;\; \text{ by } \eqref{7}\\
&= \c_0 +\b_1(m-\b_1)+ \sum_{\k \in \C_{\b_1} } \b_1(\n-m)  + \sum_{\k \notin \C_0 \cup \C_{\b_1}} \b_{\k}(\n-m)\\
& = \c_0+ \b_1(m-\b_1) +\sum_{\k \notin \C_0 } \b_{\k}(\n-m)\\
&= \c_0+\b_1(m-\b_1)  +(\n-m) \sum_{\b=1}^m \b\c_{\b} \\
&= c_0+\b_1(m-\b_1)  +(\n-m)\n \;\;\;\;\;\; \text{ by \eqref{c0c1}.}
\end{align*}
By rearranging we find that
\begin{equation}\label{newc0} \c_0 \leq  m\n-\b_1(m-\b_1)-1.
\end{equation}

We bound $\I(G)$ by maximizing $\l=\sum_{b=0}^m c_b$ subject only to Equations \eqref{c0c1}, \eqref{crcm} and \eqref{newc0}.
By \eqref{c0c1} an upper bound on $\sum_{\b=0}^m  \c_{\b}$ is given by maximizing $c_0$, maximizing $\c_{\b}$ for $\b$ small and minimizing $\c_{\b}$ for $\b$ large. Hence we substitute $\c_0=m\n-\b_1(m-\b_1)-1$ by \eqref{newc0}, substitute $\c_{\b}=0$ for $\b \notin \{0,1,\b_1,m\}$, and maximise $c_1$ and minimise $c_m$ subject to \eqref{crcm}.

First let $m=1$. Since $\b_1 \neq 0$ it follows that $\b_1=1$, and hence $c_1=d$ by \eqref{c0c1}. Now let $m \geq 2$.
Then there are three possibilities for $\b_1$. If $\b_1=m$ then to minimise $c_m$ subject to \eqref{crcm} let $c_m=2$, and so \eqref{c0c1} yields $c_1=\n-2m$. If $\b_1=1$ then $c_m=1$, and \eqref{c0c1} yields $c_1=\n-m$. Otherwise $c_m=c_{\b_1}=1$, and \eqref{c0c1} yields $c_1=\n-m-\b_1$. Hence in all cases
$$|C_1\cup C_{\b_1} \cup C_m| = 2+\n-m-\b_1.$$
Therefore
$$\sum_{\b=0}^m \c_{\b}  \leq  \big({m}\n-\b_1(m-\b_1)-1 \big) +2+\n-{m}-\b_1 = ({m}+1)\n-m+1 -\b_1(m-\b_1+1).$$
Hence if $\sum_{\b=0}^m \c_{\b}$ is maximal then $\b_1(m-\b_1+1)$ is minimal subject to $1 \leq \b_1 \leq m$. Therefore $\b_1$ is $1$ or $m$, and so
$$\sum_{\b=0}^m \c_{\b} \leq ({m}+1)\n-2{m}+1.$$
The result now follows from \eqref{max}.
\hfill \qed
\\

We now consider the lower bounds in Theorem \ref{PGLthm}.\medskip

\noindent \textit{Proof of lower bound of Theorem \ref{PGLthm}.} Let $G = \GL_n(q)$. Here we give a sequence of $m$-spaces of $V$ such that each successive point stabilizer in $G$ is a proper subgroup of its predecessor. Its length is therefore a lower bound on $\I(\PGL_n(q), \Omega)$.

For $1 \leq k \leq m\n-m^2+\n$, we define the following three variables
$$r_{k} = \Big\lfloor \frac{k-2}{m}\Big\rfloor +m+1 , \;\;\;\;\;\;\; s_{k} = m-( k-2 \bmod m), \;\;\;\;\;\;\; \text{ and }  \;t_{k}=k-m\n+m^2.$$
Notice that if $m+2 \leq k \leq m\n-m^2+1$, then 
\begin{equation}\label{ajbj} m+2 \leq r_k \leq d \;\;\;\; \text{and} \;\;\;\; 1 \leq s_k\leq m,
\end{equation}
whilst $t_k \leq d$ for all $k$ and 
$$ 2 \leq t_{k} \leq m+1  \;\;\;\;\;\; \text{ if and only if } \;\;\;\;\;\;   m\n-m^2+2 \leq k \leq m\n-m^2+m+1.$$

Hence the following sets $W_k$ of $m$ linearly independent vectors of $V$ are well defined.
$$W_{k} = \begin{cases}
\big\{ e_i \; | \; i \in \{1, \ldots, m+1\} \backslash \{m+2-k\} \big\} & \text{ for } 1 \leq k \leq m+1,\\
\big\{ e_i \; | \; i \in \{1, \ldots, m, r_{k}\} \backslash \{s_{k}\} \big\} & \text{ for } m+2 \leq k \leq m\n-m^2+1,\\
\big\{ e_1+e_{t_{k}}, e_i \; | \; i \in \{2, \ldots, m+1\} \backslash \{t_{k}\} \big\} & \text{ for } m\n-m^2+2 \leq k \leq m\n-m^2+m+1,\\
\big\{ e_1+e_{t_{k}}, e_i \; | \; i \in \{2, \ldots,m\}  \big\} & \text{ for } m\n-m^2+m+2 \leq k \leq m\n-m^2+\n.
\end{cases}$$

Let $\omega_k = \langle W_k \rangle \in \Omega$ and let $G^{(k)}=G_{\omega_1, \ldots, \omega_k}$.
For $1 \leq x,y \leq \n$ let $T(x,y)$ be the matrix $I + E_{x,y}$ and let $\sup_{x}(W_k)$ be the set of vectors in $W_k$ which are non-zero in position $x$. Recall that
$$e_iT(x,y) = \begin{cases} e_i+e_y & \text{if }i=x,\\
e_i & \text{otherwise.} \end{cases}$$
Hence if a vector $v$ is zero in position $x$, then $vT(x,y)=v$. Therefore $\omega_kT(x,y)=\omega_k$ if and only if $\sup_x(W_k)T(x,y)\subseteq \omega_k$. In particular, if $\sup_x(W_k)=\emptyset$ then $\omega_kT(x,y)=\omega_k$. Furthermore, $T(x,y) \in G$ unless $q=2$ and $x=y$.\smallskip

It is clear that $G > G\s{1}$, so let $k \in \{2, \ldots, m\n-m^2+1\}$ and let $j \leq k$. We shall show that there exist $x$ and $y$ such that $\omega_{k}T(x,y) \neq \omega_{k}$ and $\omega_{j} T(x,y)=\omega_{j}$ for all $j <k$. Hence $T(x,y) \in G\s{k-1} \backslash G\s{k}$ and so $G\s{k-1} > G\s{k}$.\medskip

First consider $k \in \{2,\ldots, m+1\},$ and let $T=T(m+1,m+2-k) $. Then $\sup_{m+1}(W_1) = \emptyset$, and for $1 < j \leq k$
$$\sup_{m+1}(W_j)T = \{e_{m+1}\}T=\{e_{m+1}+e_{m+2-k}\}. $$
Hence $\sup_{m+1}(W_j)T \subseteq \omega_j$ if and only if $j \neq k$. Therefore $\omega_{j} T = \omega_{j}$ for $j < k$, and $\omega_{k} T \neq \omega_{k}.$

Next consider $k \in \{m+2, \ldots, md-m^2+1\}$. Hence \eqref{ajbj} holds, and so we may let $T$ be the matrix $T(r_k, s_k)$.
If $j \leq m+1$ or if $r_{j} \neq r_k$, then $\sup_{r_{k}}( W_{j}) = \emptyset$ and so $\omega_{j} T= \omega_{j}.$ Therefore assume that $j \geq m+2$ and $r_{j} = r_k$. Then
$$\sup_{r_{k}}(W_j)T = \{e_{r_{k}} \}T = \{e_{r_{k}} +e_{s_{k}}\}.$$
Since $(r_j,s_j)=(r_k,s_k)$ if and only if $j=k$, it follows that $\sup_{r_k}(W_j)T \subseteq \omega_j$ if and only if $j \neq k$. Therefore $\omega_{j} T= \omega_{j}$ for $j < k$, and $\omega_{k} T\neq \omega_{k}.$
Hence $G\s{k-1} > G\s{k}$ for $1 \leq k \leq m\n-m^2+1$, and so if $q=2$ then the result follows.\medskip

It remains to consider $q>2$ and $k \geq m\n-m^2+2 $. Let $T=T(t_k,t_k)$ and let \\
$u \in \{e_i, e_1+e_i \; | \; 1 \leq i \leq \n\}$. Then
$$uT
 = \begin{cases} e_1+2e_{t_{k}} & \text{ if } u=e_1+e_{t_{k}},\\
 2u & \text{ if } u=e_{t_{k}},\\
u & \text{ otherwise.} \end{cases}$$
If $1 \leq j \leq md-m^2+1$ then $W_j \subseteq \{e_1, \ldots, e_{\n}\}$, and if $md+m^2+1 < j < k$ then $W_j \subseteq \{e_1+e_{t_j}, e_1, \ldots, e_{\n}\}$ with $t_j \neq t_k$. Hence, if $j <k$ then $\sup_{t_k}(W_j)T \subseteq \omega_j$, and so $\omega_jT = \omega_j$.
Since $e_1+e_{t_{k}}  \in \omega_{k}$ but $e_{1}+2e_{t_{k}} \notin \omega_{k}$ it follows that $\omega_{k} T(t_{k},t_{k}) \neq \omega_{k}.$ Hence $G\s{k-1} > G\s{k}$ for $1 \leq k \leq m\n-m^2+d$, and so the result follows.
\hfill \qed
\begin{remark} The interested reader may wish to check, using the notation of the previous proof, that the following holds. Let $\Lambda=(\omega_i)_{2 \leq i \leq m \n-m^2+1}$ if $q=2$, and let $\Lambda=(\omega_i)_{m+1 \leq i \leq m \n-m^2+\n}$ if $q>2$. Then $\Lambda$ is a minimal base for the action of $\PGL_{\n}(q)$ on $\PG_m(q)$. Hence
$$\B(\PGL_{\n}(q), \PG_{m}(q)) \geq \begin{cases} 
md-m^2 & \text{if $q=2$,}\\
(m+1)d-m^2-m & \text{if $q \neq 2$.}
\end{cases}$$
\end{remark}
\subsection{Upper bounds as a function of $|\Omega|$}
Our main result in this subsection is Proposition \ref{g01}, which bounds $\I(\PGL_{\n}(q), \Omega)$ as a function of $n=|\Omega|$, rather than of $m$, $\n$ and $q$. We begin by bounding the size of $\Omega=\PG_m(\F_q^d)$.

\begin{lemma}\label{omegasize} Let $\t(\n,m,q) = |\PG_m(\F_q^{\n})|$. Then
$$\log |\Omega| = \log\big( \t(\n,m,q) \big) > \begin{cases}   \frac{\n^2}{4}+\frac{1}{2} & \text{if $q=2$ and $m= \frac{\n}{2}\geq 2$,}\\
{m(\n-m)} \log q & \text{for all $m$ and $q$.}
\end{cases} $$
\end{lemma}
\begin{proof}
If either $q \neq 2$ or $m \neq \frac{\n}{2}$ then the result holds by \cite[Lemma 4.2.8]{Bianca}, so assume otherwise and let $k=\frac{d}{2}+1=m+1 \geq 3$. The statement is then equivalent to $n(2k-2,k-1,2)>2^{(k-1)^2+\frac{1}{2}}$. 

We now induct on $k$. Since $\t(4,2,2) = 35 >  2^{2^2+\frac{1}{2}}$, the result holds for $k=3$.
Now
\begin{align*} \t(2k,k,2)&=  \frac{(2^{2k}-1)(2^{2k-1}-1)(2^{2k-2}-1) \cdots (2^{k+1}-1)}{(2^{k} -1)(2^{k-1}-1)(2^{k-2}-1) \cdots (2-1)}\\
&=  \frac{(2^{2k}-1)(2^{2k-1}-1)}{(2^{k} -1)^2} \cdot \frac{(2^{2k-2}-1) \cdots (2^{k+1}-1)(2^k-1)}{(2^{k-1}-1)(2^{k-2}-1) \cdots (2-1)}\\
&= \frac{(2^{2k}-1)(2^{2k-1}-1)}{(2^{k} -1)^2} \cdot \t(2k-2,k-1,2)\\
& \geq \frac{(2^{2k}-1)(2^{2k-1}-1)}{(2^k-1)^2} \cdot 2^{(k-1)^2+\frac{1}{2}}, \;\;\;\; \text{by induction.}
\end{align*}
It is easily verified that
$$(2^{k}+1)(2^{2k-1}-1) = 2^{3k-1}+2^{2k-1}-2^{k}-1 > 2^{3k-1}-2^{2k-1}= 2^{2k-1}(2^{k}-1).$$
Hence 
$$\frac{ (2^{2k}-1)(2^{2k-1}-1)}{(2^{k}-1)^2}2^{(k-1)^2} = \frac{ (2^{k}+1)(2^{2k-1}-1)}{(2^{k}-1)} 2^{(k-1)^2} > \frac{2^{2k-1}(2^k-1)}{(2^k-1)}2^{(k-1)^2}=  2^{k^2},$$
and the result follows.
\end{proof}
Recall that $q=p^f$ with $p$ prime, and $\Omega = \PG_m(\F_q^d)$, with $n=|\Omega|$.
\begin{proposition}\label{g01} Let $G = \PGamL_{\n}(q)$ and assume that $m \leq \frac{d}{2}$. Then
$$\I(G, \Omega) \leq \begin{cases}
2(\n-1) +1 \;\;\;\;\;\;\;\; \;\;\;\;\;\;\;\; \;\;\;\;\leq 2 \log \t +1 & \text{ if $m=1$ and $q=2,$}\\
\frac{4}{3}(\n-1)\log q  +1+ \log f  \leq \frac{4}{3} \log \t +1+ \log f  & \text{ if $m=1$ and $q \geq 3,$}\\
{\frac{\n^2}{2}+ 1} \;\;\;\;\;\;\; \;\;\;\;\;\;\;\;\;\;\;\;\;\; \; \;\;\;\;\;\;\;  \leq 2 \log \t & \text{ if $m=\frac{\n}{2} \geq 2$ and $q=2$,}\\ 
2 m(\n-m)\log q+ \log f  \;\; \leq 2 \log \t+ \log f  & \text{ otherwise.}
\end{cases}$$
\end{proposition}
\begin{proof} Since $G = \PGL_{\t}(q) \rtimes C_f$, Lemma \ref{3comb}\ref{normal} and Theorem \ref{PGLthm} imply that 
\begin{equation}\label{col} \I(G) = \I(\PGL_{\t}(q))+\ell(C_f) \leq (m+1)\n-2m+1+\log f. 
\end{equation}

First let $m=1$, so that $\I(G) \leq 2(\n-1)+1+\log f$. By Lemma \ref{omegasize} $(\n-1)\log q< \log \t$. Hence the result is immediate for $q=2$, and follows from $\log q > \frac{3}{2}$ for $q \geq 3$.

Now let $m=\frac{\n}{2} \geq 2$ so that $\I(G) \leq  \frac{\n^2}{2}+1+\log f$. If $q=2$ then $ \frac{\n^2}{4} + \frac{1}{2}< \log \t $ by Lemma \ref{omegasize}, and so the result follows. If $q \geq 3$, then it follows from $\n \geq 4$ that $1 \leq \frac{\n^2}{4}$, and so 
$$\frac{\n^2}{2}+1 \leq   \frac{3\n^2}{4}  < \frac{\n^2}{2} \log q=2m(\n-m) \log q.$$
Therefore 
$$\I(G) \leq  2m(\n-m) \log q + \log f < 2 \log \t + \log f,$$
by Lemma \ref{omegasize}.
Finally consider $1< m < \frac{\n}{2}$. Then $1 \leq \n-2m$, and so 
$$\n-2m+1 \leq 2(\n-2m)\leq m(\n-2m).$$
Hence by \eqref{col}
$$\I(G)-\log f \leq m\n+\n-2m+1 \leq m\n+ m(\n-2m) = 2m(\n-m) \leq  2m(\n-m)\log q.$$
Therefore, $\I(G) \leq 2m(\n-m)\log q + \log f\leq 2\log \t + \log f,$ by Lemma \ref{omegasize}.
\end{proof}
\section{Almost simple groups}\label{ASgroups}
In this section we prove Theorem \ref{main} for almost simple groups. More precisely, we prove the following result.
\begin{theorem}\label{ASlem} Let $G$ be an almost simple primitive subgroup of $\S{n}$. If $G$ is not large base then 
$$\I(G, \Omega) <5 \log n-1.$$
\end{theorem}
We begin with two definitions which we shall use to divide this section into cases.
\begin{definition}\label{subspacesg} Let $G$ be almost simple with socle $G_0$, a classical group with natural module $V$. A subgroup $H$ of $G$ not containing $G_0$ is a \textit{subspace subgroup} if for each maximal subgroup $M$ of $G_0$ containing $H \cap G_0$ one of the following holds.
\begin{enumerate}[label=\rm{(\roman*)}]
\item $M = G_U$ for some proper non-zero subspace $U$ of $V$, where if $G_0 \neq \PSL(V)$ then $U$ is either totally singular or non-degenerate, or if $G$ is orthogonal and $p = 2$ a non-singular 1-space.
\item $G_0 = \Sp_{\n}(2^f)$ and $M \cap G_0 = \mathrm{GO}^{\pm}_{\n}(2^f)$.
\end{enumerate}
A transitive action of $G$ is a \textit{subspace action} if the point stabilizer is a subspace subgroup of $G$.
\end{definition}

\begin{definition}\label{standdefn} Let $G$ be almost simple with socle $G_0$. A transitive action of $G$ on $\Omega$ is \textit{standard} if up to equivalence of actions one of the following holds, and is \textit{non-standard} otherwise.
\begin{enumerate}[label=\rm{(\roman*)}]
\item $G_0 = \A{r}$ and $\Omega$ is an orbit of subsets or partitions of $\{1, \ldots , r\}$.
\item $G$ is a classical group in a subspace action.
\end{enumerate}
\end{definition}
This section is split into three subsections. The first considers $G_0 = \PSL_{\n}(q)$ acting on subspaces and pairs of subspaces. In the second, we deal with the case of $G$ another classical group in a subspace action. Finally in the third we prove Theorem \ref{ASlem}.
\subsection{$G_0 = \PSL_{\n}(q)$}
Let $G$ be almost simple with socle $\PSL_d(q)$, in a subspace action on a set $\Omega$. We shall use the results in Section \ref{pglmspaces} to bound $\I(G, \Omega)$.

We begin with a preliminary lemma.

\begin{lemma}\label{logfbound} Let $1 \leq m \leq \frac{\n}{2}$, let $p$ be a prime, let $f \geq 1$ and let $q=p^f$. Then
$$ m(\n-m) \log q \geq
 \begin{cases}
   \log f + 1& \text{ if } m=1,\\
 3\log f+4 & \text{ if } m \geq 2.
  \end{cases}
$$
\end{lemma}
\begin{proof}
Let $m=1$. Then
$$m(\n-m)\log q = (\n-1) \log q = (\n-1)f\log p \geq f \geq \log f+1.$$
Now let $m \geq 2$, so that $m(\n-m) \geq 4$. Then
$$m(\n-m)\log q \geq 4 \log q \geq 3 \log q +1=3 f \log p +1 \geq 3f+1 \geq 3(\log f +1)+1=3\log f +4$$
as required.
\end{proof}

\begin{proposition}\label{newlem} Let $G$ be almost simple with socle $\PSL_{\n}(q)$ acting on $\Omega = \PG_m(V),$ and let $\t = |\Omega|$. Then 
$$\I(G) < 3 \log \t.$$
\end{proposition}
\begin{proof}
If $m=1$ then $G \leq \PGamL_{\n}(q)$, so $\I(G) \leq \I(\PGamL_{\n}(q))$ by Lemma \ref{3comb}\ref{subgroup}. Otherwise $G \cap \PGamL_{\n}(q)$ has index at most 2 in $G$, so by Lemma \ref{3comb}\ref{subgroup} and \ref{2.3.8}
$$\I(G) \leq \I(G\cap \PGamL_{\n}(q))+1 \leq  \I( \PGamL_{\n}(q))+1.$$
Therefore we can bound $\I(G)$ by our bound for $\I(\PGamL_{\n}(q))$ when $m=1$, and by one more than that when $m >1$. Thus Proposition \ref{g01} yields
$\I(G) \leq 2 \log \t + \log f  +1$.
Combining Lemmas \ref{omegasize} and \ref{logfbound} gives $\log f +1 \leq m(\n-m)\log q <  \log \t$, hence the result holds.
\end{proof}

We now consider the action of $G$ on the following subsets of $\PG_m(V) \times \PG_{\n-m}(V)$, with $m< \frac{d}{2}:$
\begin{align*}
 \Omega^{\oplus} =& \Big\{ \{U,W\} \; \Big| \; U,W \leq V, \; \dim U =m, \; \dim W = \n-m \text{ with } U \oplus W=V \Big\} ,\\
 \Omega^{\leq} =& \Big\{ \{U,W\} \; \Big| \; U,W \leq V,  \; \dim U =m, \; \dim W = \n-m \text{ with } U  \leq W \Big\}.
\end{align*}
Note that in both cases we require $\n \geq 3$.
\begin{lemma}\label{pairslem} 
Let $G$ be almost simple with socle $\PSL_{\n}(q)$, let $H = G \cap \PGamL_{\n}(q)$ and let $\Omega$ be either $\Omega^{\oplus}$ or $\Omega^{\leq}$. Then 
$$\I(G, \Omega) \leq 2\I(H, \PG_m(V)) +1.$$
\end{lemma}
\begin{proof} 
We first show that
\begin{equation}\label{edit2} \I(H, \Omega) \leq \I(H,\PG_m(V))+\I(H,\PG_{\n-m}(V)).
\end{equation}

Let $\l=\I(H, \Omega)$ and let $\Lambda = \big(\{U_1, W_1 \}, \ldots , \{U_{\l}, W_{\l} \}\big)$ be a corresponding base, where $\dim(U_i)=m$ for all $i$. Let $\Pi=(U_1, \ldots, U_{\l} ) $ and let $\Sigma = (W_1, \ldots, W_{\l} ) $. Then by Lemma \ref{2.2.6} $\Pi$ and $\Sigma$ contain subsequences which can be extended to irredundant bases for the action of $H$ on $\PG_m(V)$ and $\PG_{d-m}(V)$ respectively.

Let $\Pi'$ be the subsequence of $\Pi$ which contains $U_i$ if and only if $H_{U_1, \ldots, U_{i-1}} > H_{U_1, \ldots, U_{i-1}, U_i}$. Then $\Pi'$ can be extended to an irredundant base for the action of $H$ on $\PG_m(V)$. Let $k=|\Pi'|$, so $\k \leq \I(H, \PG_m(V)).$

Let $\Sigma' = (W_{j_1}, \ldots, W_{j_{(l-k)}})$ be the subsequence of $\Sigma$ which contains $W_i$ if and only if $H_{U_1, \ldots, U_{i-1}} = H_{U_1, \ldots, U_{i-1}, U_i}$. 
Assume, for a contradiction, that $\Sigma'$ cannot be extended to an irredundant base for the action of $H$ on $\PG_{\n-m}(V)$. Since $H$ is irreducible, $H> H_{W_{j_1}}$. Therefore there exists $s \geq 2$ such that 
$$H_{W_{j_1},W_{j_2}, \ldots, W_{j_{(s-1)}}} = H_{W_{j_1}, W_{j_2},\ldots, W_{j_{(s-1)}}, W_{j_s}}.$$
Let $i=j_s$. Then intersecting both sides of the above expression with $H_{W_1, \ldots, W_{i-1}}$ gives
\begin{equation}\label{Weqn} H_{W_1, \ldots,  W_{i-1}} = H_{W_1, \ldots, W_{i-1},W_{i}}. \end{equation}
Since $W_{i} \in \Sigma'$ it follows that 
\begin{equation}\label{Ueqn} H_{U_1, \ldots,  U_{i-1}} = H_{U_1, \ldots, U_{i-1},U_{i}}. \end{equation}
Elements of $H = G \cap \PGamL_{\n}(q)$ cannot map $U_i$ to $W_i$. Therefore \eqref{Weqn} and \eqref{Ueqn} imply that 
$$H_{\{U_1, W_1\}, \ldots, \{U_{i-1}, W_{i-1}\}} = H_{\{U_1, W_1\}, \ldots, \{U_{i-1}, W_{i-1}\}, \{U_{i}, W_{i}\}},$$ 
a contradiction since $\Lambda$ is irredundant. Hence $l-k \leq \I(H,\PG_{n-m}(V))$, and so \eqref{edit2} holds.

The subgroups of $\Sym{\PG_m(V)}$ and $\Sym{\PG_{n-m}(V)}$ representing the actions of $H$ are permutation isomorphic. Therefore \eqref{edit2} implies that $\I(H, \Omega) \leq 2\I(H, \PG_m(V))$. Since $H$ has index at most 2 in $G$, the result follows from Lemma \ref{3comb}\ref{2.3.8}.
\end{proof}
\begin{lemma}\label{g02} Let $\Omega$ be either $\Omega^{\oplus}$ or $\Omega^{\leq}$, and let $\t=|\Omega|$. Let $G$ be an almost simple subgroup of $\Sym{\Omega}$ with socle $\PSL_{\n}(q)$. Then
$$\I(G) <  5 (\log \t-1) .$$
\end{lemma}
\begin{proof}
Let $H = G \cap \PGamL_{\n}(q)$, then by Proposition \ref{g01} and Lemma \ref{pairslem} 
\begin{equation}\label{14} \I(G) \leq   2 \I(H, \PG_{m})+1  \leq  \begin{cases}
4(\n-1) +3  & \text{ if } m=1 \text{ and } q=2,\\
\frac{8}{3}(\n-1) \log q   + 2 \log f +3 & \text{ if } m=1 \text{ and } q \geq 3,\\
4 {m(\n-m)} \log q + 2\log f +1 & \text{ otherwise}.
\end{cases}
\end{equation}
Since $\t \geq 2|\PG_m(V)|$, Lemma \ref{omegasize} gives 
\begin{equation}\label{2t}
m(\n-m) \log q < \log \frac{\t}{2} =\log \t -1.
\end{equation}
Recall that $d \geq 3$.
First let $m=1$. If $(\n,q)=(3,2)$, then $n \in \{21,28\}$. Therefore by \eqref{14} it follows that $\I(G) \leq 11< 5(\log n-1)$.
Hence if $q=2$ then we may assume that $\n \geq 4$, and so by \eqref{14} and \eqref{2t}
$$\I(G) \leq 4(\n-1)+3 \leq 5(\n-1) < 5 (\log \t-1).$$
Still with $m=1$, let $q \geq 3$. Then
$$\begin{array}{rll}
\I(G) & \leq \frac{8}{3}(\n-1) \log q   + 2 \log f +3 & \text{ by \eqref{14},}\\
& \leq \frac{8}{3}(\n-1) \log q +2(\n-1) \log q+ 1 & \text{ by Lemma \ref{logfbound},}\\
& <  5(\n-1)\log q & \text{ since $1 < \frac{1}{3} (\n-1)\log q $, } \\
& < 5 (\log \t -1)& \text{ by \eqref{2t}.}
\end{array}$$

Finally, let $m \geq 2$. Then $2 \log f +1 \leq m(\n-m)\log q$ by Lemma \ref{logfbound}. Hence by \eqref{14} and \eqref{2t} it follows that
$$\I(G) \leq 4 {m(\n-m)}  \log q +2 \log f+1 \leq 5 {m(\n-m)}  \log q < 5(\log \t -1) .$$
\end{proof}
\subsection{$G_0$ another classical group}

 \begin{lemma}\label{nov2} Let $G$ be almost simple with socle $G_0 = \mathrm{P}\Omega_8^+(q),$ acting faithfully and primitively on a set $\Omega$ of size $n$. Then
 $$\I(G) < 5\log \t-1.$$
 \end{lemma}
 \begin{proof} 
Let $q \geq 3$. Then the reader may check that $6f<q^{2}$, and so by \cite[(6.19)]{GLS} 
$$|G|<6fq^{28} \leq q^{30}.$$
If $q=2$, then $|G|  \leq 6|G_0| < q^{30}$ also. Hence by Lemma \ref{fands}\ref{sneaky}, since $n >4$,
$$\I(G) \leq \log q^{30}-1=5 \log q^6-1<5 \log n-1,$$
by {\cite[(6.20)]{GLS}}. 
\end{proof}
\begin{proposition}\label{otherclassical} Let $G$ be almost simple with socle $G_0$, a classical group with natural module $V$. Assume that $G_0 \neq \PSL(V)$ and $G_0 \neq \mathrm{P}\Omega_8^{+}(q)$. Let $0 < m < \n$, let $\Omega$ be a $G$-orbit of totally isotropic, totally singular, or non-degenerate subspaces of $V$ of dimension $m$, and let $\t=|\Omega|$. Then 
$$\I(G, \Omega) < 5 \log \t-1.$$
\end{proposition}
\begin{proof}
First let $G_0 = \mathrm{P}\Omega_d^{+}(q)$ and $m = \frac{\n}{2}$. Then $\n \geq 10$, and so $2\n^2-12\n-16>0$. Hence $10d^2-20d>8d^2-8d+16$ and it follows that $\frac{d^2}{8}-\frac{d}{4} > \frac{d^2}{10}-\frac{d}{10}+\frac{1}{5} $. 
By \cite[Table 4.12 ]{tim1} 
\begin{equation}\label{prod} \t= \prod_{i=1}^{\frac{\n}{2}-1} (q^i+1)>\prod_{i=1}^{\frac{\n}{2}-1}q^i=q^{\frac{d^2}{8}-\frac{d}{4}} > q^{ \frac{d^2}{10}-\frac{d}{10}+\frac{1}{5} }.
\end{equation}
Hence 
$$\begin{array}{lll}
\I(G) & \leq \log |G|-1 & \text{ by Lemma \ref{fands}\ref{sneaky},}\\
& \leq \log \Big(q^{\frac{\n^2}{2}-\frac{\n}{2}+1}\Big)-1 & \text{ by \cite[p25]{GLS}},\\
& = 5\log \Big(q^{\frac{\n^2}{10}-\frac{\n}{10}+\frac{1}{5}}\Big)-1 \\
& < 5 \log \t -1 & \text{ by \eqref{prod}.}
\end{array}$$
Therefore we may assume for the rest of the proof that $G_0 \neq \POm_d^{+}(q)$, so by \cite[Lemma 7.14]{GLS}
\begin{equation}\label{smallom}
\frac{1}{2}m(\n-m)\log q < \log \t.
\end{equation}
Since $\Omega$ is a $G$-orbit of subspaces, if $G_0=\PSp_4(q)$ then $G$ does not induce the graph isomorphism by \cite[Table 8.14]{ColvaBook}. Hence since $G_0 \neq \mathrm{P}\Omega_8^{+}(q)$ and $\Omega \subseteq \PG_m(V)$, we may assume that
$G \leq \PGamL_{\n}(q)$. Then Lemma \ref{3comb}\ref{subgroup} implies that
$$\I(G) \leq \I\big(\PGamL_{\n}(q),\PG_m(V)\big),$$
and so in particular the bounds from Proposition \ref{g01} apply.

We begin with $m=1$. If $q=2$ then we split into two cases. If $\n \leq 4$ then by Proposition \ref{g01}
$$\I(G) \leq 2(d-1) +1\leq 7 < 5 \log \t -1.$$ 
If instead $\n \geq 5,$ so that $\frac{1}{2}(d-1) \geq 2$, then by Proposition \ref{g01} and \eqref{smallom}
$$\I(G) \leq 2(\n-1)+1 \leq 2(d-1)+\frac{1}{2}(d-1)-1< 5 \log \t -1.$$
To complete the case of $m=1$, let $q \geq 3$. Since $G_0 \neq \PSL_d(q)$ we may assume that $d \geq 3$ and so it can be verified that $ \frac{6}{7f}(2+\log f)+1<d$. Hence
$$2+\log f< \frac{7}{6}f(d-1) \leq \frac{7}{6}f(d-1)\log p= \frac{7}{6}(d-1)\log q .$$
Therefore it follows from  Proposition \ref{g01} and \eqref{smallom} that 
$$\I(G) \leq \frac{4}{3}(d-1)\log q+\log f+1 < \frac{5}{2}(d-1)\log q-1<5 \log n-1.$$

Now let $m = \frac{\n}{2} $ and $q=2$. Then by Proposition \ref{g01} and \eqref{smallom}
$$\I(G) \leq \frac{\n^2}{2}+1 =4\Big(\frac{1}{2}m(d-m) \Big)+1<4 \log \t+1<5\log n-1,$$
since $n>4$.

Hence we may assume that $m \geq 2$, and that if $m = \frac{\n}{2}$ then $q \geq 3$. Therefore 
$$\begin{array}{lll}
\I(G) & \leq 2m(\n-m)\log q+\log f & \text{ by Proposition \ref{g01},}\\
& \leq 2m(\n-m)\log q+\frac{1}{3}m(d-m)\log q - \frac{4}{3} & \text{ by Lemma \ref{logfbound},}\\
& < \frac{14}{3} \log \t - \frac{4}{3} & \text{ by \eqref{smallom},}\\
& < 5 \log \t -1.
\end{array}$$
\end{proof}
\subsection{Proof of Theorem \ref{ASlem}}
We begin by proving Theorem \ref{ASlem} for non-standard actions. 

\begin{proposition}\label{nonstand} Let $G$ be an almost simple, primitive non-standard subgroup of $\Sym{\Omega}$ and let $n=|\Omega|$. Then 
$$\I(G, \Omega) \leq 4 \log \t+1.$$
\end{proposition}
\begin{proof}
By a landmark result of Burness and others \cite{tim, tim1, tim2, tim3}, either $(G, \Omega) = (M_{24}, \{1, \ldots, 24\})$ or $\bb(G, \Omega) \leq 6$. 
By {\cite[p10]{GLS}}, $\I(\M_{24}, \{1, \ldots, 24\})=7 < 2\log 24$. If $\bb(G) \leq 5$, then the result follows by Lemma \ref{fands}\ref{fourth}. Hence we may assume that $\bb(G, \Omega)=6$. \medskip

Let $G$ have point stabilizer $H$. By a further result of Burness \cite[Theorem 1]{newtim2}, either
\begin{equation}\label{GH} (G,H) \in \big\{ (\M_{23},\M_{22}), (\mathrm{Co}_3,\mathrm{McL}.2), (\mathrm{Co}_2,\mathrm{U}_6(2).2), (\mathrm{Fi}_{22}.2,2.\mathrm{U}_6(2).2)\big\} \text{ or}
\end{equation}
\begin{equation}\label{G0H} (\Soc(G),H)\in \big\{(E_7(q),P_7), (E_6(q),P_1), (E_6(q),P_6)\big\}.
\end{equation}

Therefore $H$ is insoluble, and so $\I(H) +1 < \log |H|$ by Lemma \ref{fands}\ref{insol}. Hence if we can prove in each case that $|H| < [G:H]^4$ then the result will follow, since
$$\I(G)=\I(H)+1 <\log|H| < \log [G:H]^4 = 4 \log \t.$$
Since $\M_{23}$ is the point stabilizer of $\M_{24}$ it follows that
$$\I(\M_{23}, \{1,\ldots,23\})=\I(\M_{24}, \{1,\ldots,23,24\}) -1=6<2\log 23.$$
If $G$ is $\mathrm{Co}_3$, $\mathrm{Co}_2$ or $\mathrm{Fi}_{22}.2$, then using \cite{Atlas} it is easily verified that $|H| < [G:H]^4$, so the result follows.

Therefore assume that $(\Soc(G),H)$ is as in \eqref{G0H}. Let $m(G)$ be the smallest degree of a faithful transitive permutation representation of $G$.
If $|G| < m(G)^5$, then 
$$|H |= \frac{|G|}{[G:H]} <\frac{m(G)^5}{[G:H]}\leq [G:H]^4,$$
and so the result will follow.

First let $G_0 = E_6(q)$. By \cite{Atlas}
$$|E_6(q)|=\frac{q^{36}(q^{12}-1)(q^{9}-1)(q^{8}-1)(q^{6}-1)(q^{5}-1)(q^{2}-1)}{(3,q-1)}$$
and $|\Out(E_6(q))| \leq 2f(3,q-1) \leq q(3,q-1)$. Hence
$$|G| \leq q^{37}(q^{12}-1)(q^{9}-1)(q^{8}-1)(q^{6}-1)(q^{5}-1)(q^{2}-1)<q^{37+12+9+8+6+5+2}=q^{79}.$$
By \cite[p2]{AV} 
$$m(G) \geq m(G_0) \geq \frac{(q^9-1)(q^8+q^4+1)}{q-1}=(q^8+q^7+ \cdots + q+1)(q^8+q^4+1)>q^{16}.$$
Hence $|G|<q^{79}<q^{80} < m(G)^5$ as required.\smallskip

Now let $G_0=E_7(q)$. By \cite{Atlas}
$$|E_7(q)|=\frac{q^{63}(q^{18}-1)(q^{14}-1)(q^{12}-1)(q^{10}-1)(q^{8}-1)(q^{6}-1)(q^{2}-1)}{(2,q-1)}$$
and $|\Out(E_7(q))|=(2,q-1)f<(2,q-1)q$. Hence
$$|G| \leq q^{64}(q^{18}-1)(q^{14}-1)(q^{12}-1)(q^{10}-1)(q^{8}-1)(q^{6}-1)(q^{2}-1) <q^{64+18+14+12+10+8+6+2}=q^{134}.$$
By \cite[p5]{AV}
$$m(G) = \frac{(q^{14}-1)(q^9+1)(q^5+1)}{q-1}
 = (q^{13}+q^{12} + \cdots +q+1)(q^9+1)(q^5+1)
>q^{13+9+5}
=q^{27}.$$
Hence $|G| <q^{134}<q^{135} <m(G)^5$.
\end{proof}
We note that this bound could be improved if the groups with minimal base size 5 were classified. \smallskip

\noindent \textit{Proof of Theorem \ref{ASlem}.} If the action of $G$ on $\Omega$ is non-standard then the result follows by Proposition \ref{nonstand}. Hence we may assume that $G$ is standard. 

If $G$ is alternating and not large base, then $\Omega$ is a set of partitions. Hence $\I(G, \Omega) < 2 \log |\Omega|$ by \cite[Lemma 6.6]{GLS}. 

Therefore $G$ is classical, and the action of $G$ on $\Omega$ is a subspace action.
If $G$ is as in Case (ii) of Definition \ref{subspacesg}, then $\I(G, \Omega) < \frac{11}{3} \log |\Omega|$ by \cite[Lemma 6.7]{GLS}.
If $G_0=\PSL_{\n}(q)$ and $\Omega$ is a set of subspaces, or a set of pairs of subspaces, of $V$, then the result holds by Proposition \ref{newlem} and Lemma \ref{g02} respectively. 
If $G_0 \neq \PSL_d(q)$ and $\Omega$ is a set of subspaces, then the result follows by Proposition \ref{otherclassical}.
Hence by {\cite[5.4]{newtim}} we may assume that either $G_0=\mathrm{P}\Omega_8^+(q)$ and $G$ contains a triality automorphism; or $G_0=\textrm{Sp}_4(2^f)'$ and $G$ contains a graph automorphism. In the former case the result holds by Lemma \ref{nov2}. In the latter $\I(G, \Omega) < \frac{11}{3} \log |\Omega|$ by \cite[Lemma 6.12]{GLS}. \hfill \qed
\section{Proof of Theorem \ref{main}}\label{PAgroups}
Here we use the form and notation of the O'Nan Scott Theorem from \cite{cheryl}.
We begin by considering groups of type PA, and then we prove Theorem \ref{main}.
\begin{lemma}\label{PAlem} Let $G$ be a subgroup of $\S{\t}$ of type PA that is not large base. Then 
$$\I(G) < 5\log \t.$$
\end{lemma}
\begin{proof} 
Since $G$ is of type PA there exists an integer $r \geq 2$, a finite set $\Delta$ and an almost simple subgroup $H$ of $\Sym{\Delta}$ such that $G \leq H \wr \S{r}$. Since $G$ is not large base, neither is $H$. Let $s=|\Delta|$, so that $n=s^r$ with $s \geq 5$. Then
$$\begin{array}{rll}
\I(G, \Omega) & \leq \I(H^r , \Delta^r) + \ell(\S{r}) & \text{ by Lemma \ref{3comb}\ref{subgroup} and \ref{normal},}\\
& \leq \I(H^r , \Delta^r) +  \frac{3}{2}r & \text{ by \cite{PC},}\\
& \leq r(\I(H, \Delta)-1)+1+\frac{3}{2}r & \text{ by \cite[Lemma 2.6]{GLS},}\\
&<  r(5 \log s-2)+1+\frac{3}{2}r & \text{ by Theorem \ref{ASlem},}\\
& < 5\log s^r -\frac{1}{2}r+1 \\
& \leq 5 \log \t & \text{ since } r \geq 2.
\end{array}$$
\end{proof}
We can now prove Theorem \ref{main}.\\

\noindent \textit{Proof of Theorem \ref{main}}
Let $G$ be a primitive group which is not large base. If $G$ is almost simple, then the result holds by Theorem \ref{ASlem}. If $G$ is of type PA then the result holds by Lemma \ref{PAlem}. For all other $G$, the result holds by \cite[Props 3.1, 4.1 and 5.1]{GLS}
 \hfill \qed

\smallskip
Veronica Kelsey \& Colva M. Roney-Dougal: Mathematical Institute, Univ. St Andrews, KY16 9SS, UK\\
vk49@st-andrews.ac.uk\\
Colva.Roney-Dougal@st-andrews.ac.uk
\end{document}